\documentclass[10pt,a4paper,oneside,notitlepage]{amsart}
\usepackage[english]{babel}
\usepackage[latin1]{inputenc}
\usepackage{amsmath}
\usepackage{fancyhdr}
\usepackage{amssymb}
\usepackage{amsthm}
\usepackage{sidecap}
\usepackage{booktabs}
\usepackage{graphicx}
\usepackage{float}
\usepackage{dsfont}
\usepackage{mathtools}
\usepackage{indentfirst}
\usepackage{frontespizio}
\usepackage{mathrsfs}
\usepackage{titlesec}
\usepackage{listings}

\usepackage{abstract}
\addto\captionsenglish{}

\makeatletter
\def\blfootnote{\xdef\@thefnmark{}\@footnotetext}
\makeatother

\numberwithin{equation}{section}

\titleformat{\section}
  {\center\normalfont\scshape}{\thesection.}{0.5em}{}

\titleformat{\subsection}[runin]
  {\normalfont\bfseries\filright}{\thesubsection}{0.5 em}{}

\usepackage{hyperref}
\hypersetup{
    colorlinks=false,
    linkcolor=blue ,
    filecolor=magenta,      
    urlcolor=cyan,
}

\usepackage{enumitem}
\newlist{Aenumerate}{enumerate}{1}
\setlist[Aenumerate]{label=A.\arabic*}
  
\newcommand{\R}{\mathbb{R}}

\newcommand{\N}{\mathbb{N}}
\newcommand{\C}{\mathbb{C}}

\newcommand{\D}{\mathscr{D}}
\newcommand{\G}{\mathscr{G}}

\DeclarePairedDelimiter{\abs}{\lvert}{\rvert}
\theoremstyle{definition}
\newtheorem*{notazione}{Notation}

\newtheorem{definizione}{Definizione}[section]
\newtheorem{Teorema}[definizione]{Theorem}

\newtheorem{cor}[definizione]{Corollary}
\newtheorem{Lemma}[definizione]{Lemma}
\newtheorem{oss}[definizione]{Remark}

\newcommand*{\sump}{} 
\DeclareRobustCommand*{\sump}{%
    \mathop{{\sum}^{\mathrlap{*}}}%
}
  
\begin{document}

\thispagestyle{empty}
\begin{center}
\begin{large}\textbf{ON A BASIC MEAN VALUE THEOREM WITH EXPLICIT EXPONENTS}\end{large} \\ \hspace{0.2cm} \\
\begin{small}
\textsc{Matteo Ferrari}
\end{small}
\begin{abstract}
\noindent \textsc{Abstract}. In this paper we follow a paper from A. Sedunova \cite{Sedunova20182} regarding R. C. Vaughan's basic mean value Theorem \cite{Vaughan1980} to improve and complete a more general demonstration for a suitable class of arithmetic functions as started by A. C. Cojocaru and M. R. Murty  \cite{CojocaruMurty2006}. As an application we derive a basic mean value Theorem for the von Mangoldt generalized functions.  \blfootnote{\hspace{-0.55cm} \textit{Date:} \today  }
\vspace*{1.5em}
\end{abstract}
\end{center}

\section{Introduction}
In 1980 R. C. Vaughan \cite{Vaughan1980} proved the basic mean value Theorem
\begin{Teorema} \textit{
\begin{equation*}
\sum_{q \le Q} \frac{q}{\phi(q)} \sump_{\chi \bmod q} \max_{y \le x} \biggl \vert \sum_{n \le y} \Lambda(n) \chi(n) \biggr\vert \ll (x + x^{\frac{5}{6}}Q + x^{\frac{1}{2}}Q^2)\log^4x,
\end{equation*}
where $\Lambda$ is the von Mangoldt function and the sum is restricted to primitive characters.}
\end{Teorema}
This result was a major tool for R. C. Vaughan to prove with elementary methods the Bombieri-Vinogradov Theorem. Recently A. Sedunova \cite{Sedunova20182} improved the exponent of the logarithm using a weighted version of Vaughan's identity and an estimate due to M. B. Barban and P. P. Vehov \cite{BarbanVehov1968} related to Selberg's sieve. 
A. C. Cojocaru and M. R. Murty in \cite{CojocaruMurty2006} proved a more general Theorem than the basic mean value Theorem. We will follow their proof improving the results adapting Sedunova's method. 
Using the main Theorem \ref{mainresult} we will be able to prove a basic mean value Theorem for the generalized von Mangoldt function $\Lambda_k = \mu \star \log^k$, precisely
\begin{Teorema} \textit{ For each $k \in \N$, $\epsilon>0$ it holds
\begin{equation*}
\sum_{q \le Q} \frac{q}{\phi(q)} \sump_{\chi \bmod q} \max_{y \le x} \Bigl\vert \sum_{n \le y} \Lambda_k(n) \chi(n) \Bigr\vert \ll_k \bigl(x + x^{\frac{13}{14}+\epsilon}Q + x^{\frac{1}{2}}Q^2 \bigr) \log^{k+1}x.
\end{equation*}}
\end{Teorema}

\begin{notazione}
Given $A \subset \R$, with $\mathds{1}_{A}$ we denote the characteristic function of $A$, when we write $\mathds{1}$ we suppose $A = \{1\}$. Given an arithmetic function $f : \N \to \C$ and two real numbers $U<V$, we write $f_{\le U}$ for $f \cdot \mathds{1}_{[1,U]}$, $f_{>V}$ for $f \cdot (1 - \mathds{1}_{[1,V]})$ and with $f_{(U,V]}$ for $f \cdot \mathds{1}_{(U,V]}$. We use the standard Vinogradov notation $\ll$ and when the implicit constant does depend on something we specify it. The quantities $Q, M_1, M_2, N_1, N_2$ are always some functions that depend on $x$, when we use the $\ll$ notation we assume $x \to +\infty$. 
\end{notazione}

\section{Main result}
Let us indicate the class of arithmetic functions
\begin{equation} \label{defD}
\mathscr{D} = \Bigl\{ D:\N \to \C : \sum_{n \le x} \abs*{D(n)}^2 \ll x\log^{\alpha}x \, \,  \, \textit{for some} \hspace{0.13cm} \alpha \ge 0 \Bigr\}
\end{equation}
and for $D \in \mathscr{D}$ let
\begin{equation} \label{alpha}
\alpha_D = \inf \biggl\{ \alpha \ge 0 : \sum_{n \le x} \abs*{D(n)}^2 \ll x \log^{\alpha}x\biggr\}.
\end{equation}
We will need also information about the average of $\abs*{D(n)}/n^k$ for $k \in [0,1]$. Let us indicate
\begin{equation} \label{beta}
\beta_D(k) = \inf \biggl\{ \beta \ge 0 : \sum_{n \le x} \frac{\abs*{D(n)}}{n^{k}} \ll_k x^{1-k} \log^{\beta}x \biggr\}.
\end{equation}
It is straightforward that if $D \in \D$ then $\beta_D(k) < + \infty$ for all $k \in [0,1]$, we will give a precise bound in Lemma \ref{prop12}. 

From now on we consider two arithmetic functions $f, g : \N \to \C$ with $f(1) \ne 0$. We define $\mu_f, \Lambda_{fg}$ as
\begin{align} 
 \label{convo1} \mathds{1} & = \mu_f \star f,   \\
\label{convo2} \Lambda_{fg} & = \mu_f \star g. \end{align}
In particular $\mu_f$ is the convolution inverse of $f$: it exists and is unique since $f(1) \ne 0$.
We can understand better these definitions with the help of the associated formal Dirichlet series: if
\begin{equation*}
G(s) = \sum_{n \ge 1} \frac{g(n)}{n^s}, \quad F(s) = \sum_{n \ge 1} \frac{f(n)}{n^s};
\end{equation*} 
then
\begin{equation*}
\frac{G(s)}{F(s)} = \sum_{n \ge 1} \frac{\Lambda_{fg}(n)}{n^s}, \quad \frac{1}{F(s)} = \sum_{n \ge 1} \frac{\mu_f(n)}{n^s}.
\end{equation*}
The benchmark case is clearly when
\begin{equation*}
f=1, \quad g=\log,\quad \mu_f=\mu, \quad \Lambda_{fg}=\Lambda.
\end{equation*}

We are interested in estimates for
\begin{equation*}
\sum_{q \le Q} \frac{q}{\phi(q)} \sump_{\chi \bmod q} \max_{y \le x} \Bigl\vert \sum_{n \le y} \Lambda_{fg}(n) \chi(n) \Bigr \vert.
\end{equation*}
We have two trivial bounds. Using the triangle inequality we obtain, for each $\epsilon>0$,
\begin{align}
\sum_{q \le Q} \frac{q}{\phi(q)} \sump_{\chi \bmod q} \max_{y \le x} \Bigl\vert \sum_{n \le y} \Lambda_{fg}(n) \chi(n) \Bigr \vert & \le
\sum_{q \le Q} q \,  \sum_{n \le x} \abs*{\Lambda_{fg}(n)} \nonumber \\ & \ll x Q^2 \log^{\beta_{\Lambda_{fg}}(0)+\epsilon}x.  \label{triv1}
\end{align}
Using the Cauchy-Schwarz inequality we obtain, for each $\epsilon>0$,
\begin{align}
\sum_{q \le Q} \frac{q}{\phi(q)} \sump_{\chi \bmod q} \max_{y \le x} \Bigl\vert \sum_{n \le y} \Lambda_{fg}(n) \chi(n) \Bigr \vert \nonumber & \le
\sum_{q \le Q} q \, \Bigl( \sum_{n \le x} \abs*{\Lambda_{fg}(n)}^2 \Bigr)^{\frac{1}{2}}\Bigl( \sum_{n \le x} 1 \Bigr)^{\frac{1}{2}} \\ & \ll x Q^2 \log^{\frac{\alpha_{\Lambda_{fg}}}{2}+\epsilon}x. \label{triv2}
\end{align}
We can improve these inequalities assuming further hypotheses for $f, g,\mu_f$ and $\Lambda_{fg}$.
\begin{Teorema} \label{mainresult}
 \textit{
We suppose that $g, f, \mu_f$ and $\Lambda_{fg}$, as defined before, satisfy the following hypotheses:
\begin{enumerate}[leftmargin=*,label=\textsc{(H\arabic*)}]
\item \label{H1} $ g : \N \to \R^+$ is an increasing function;
\item \label{H2} $f, \, \mu_f, \, \Lambda_{fg} \, \in \mathscr{D}$;
\item \label{H3} there exist $\theta_f, \gamma_f \in [0,1]$ such that, for any non-principal primitive Dirichlet character $\chi \bmod q$
\begin{equation*}
\sum_{n \le x} f(n) \chi(n) \ll x^{\theta_f}q^{\frac{1}{2}} \log q + x^{\gamma_f};
\end{equation*}
\item \label{H4} for each $1 \le V_1 < V_2 $ there exists a bounded function $\eta(b)=\eta(b;V_1,V_2)$ such that $\eta(b)=1$ for $b \le V_1$, $\eta(b)=0$ for $b > V_2$ and 
\begin{equation*}
\sum_{n=1}^V \Bigl\vert \bigl((\mu_f \cdot \eta) \star f \bigr)(n) \Bigr\vert^2 \ll \frac{V} {\log (\frac{V_2}{V_1})}.
\end{equation*} 
\end{enumerate}
Then for each $\epsilon>0$, $U_0 \le U_1, V_1 < V_2$ it holds
\begin{equation*}
\sum_{q \le Q} \frac{q}{\phi(q)} \sump_{\chi \bmod q} \max_{y \le x} \Bigl\vert \sum_{n \le y} \Lambda_{fg}(n) \chi(n) \Bigr\vert \ll H(x,Q,U_0,U_1,V_1,V_2)
\end{equation*}
and we have
\begin{align*}
H(x,Q,U_0,U_1,V_1,V_2) & \ll  U_1Q^2 \log^{\beta_{\Lambda_{fg}}(0)+\epsilon}U_1 \\
 & + x^{\theta_f} (U_0V_2)^{1-\theta_f}Q^{\frac{5}{2}} \log^{\beta_{\mu_f}(\theta_f) + \beta_{\Lambda_{fg}}(\theta_f) + 1+\epsilon}(U_0V_2Q) \\ & + x^{\gamma_f} (U_0V_2)^{1-\gamma_f}  Q^2 \log ^{\beta_{\mu_f}(\gamma_f) + \beta_{\Lambda_{fg}}(\gamma_f)+\epsilon}(U_0V_2) \\ & + x  \log^{\beta_f(0) + \beta_{\mu_f}(1) + \beta_{\Lambda_{fg}}(1)+\epsilon}(xU_0V_2) \\
& + \biggl(\bigl(x^{\frac{1}{2}}Q^2 + x\bigr)\log U_1 + x^{\frac{1}{2}}Q\Bigl(U_1^{\frac{1}{2}} + \frac{x^{\frac{1}{2}}}{U_0^{\frac{1}{2}}}\Bigr)\biggr)\frac{\log^{\frac{\alpha_{\Lambda_{fg}}}{2}+1+\epsilon}x}{\log^{\frac{1}{2}}(\frac{V_2}{V_1})} \\ 
 & + g(x) V_2 Q^{\frac{5}{2}} \log^{\beta_{\mu_f}(0)+1+\epsilon}(V_2Q) +  g(x) x \log^{\beta_{\mu_f}(1)+\epsilon}V_2 \\
 & + \biggl(\bigl(x^{\frac{1}{2}}Q^2 + x\bigr)\log{x} + xQ\Bigl(\frac{1}{V_1^{\frac{1}{2}}} + \frac{1}{U_1^{\frac{1}{2}}}\Bigr)\biggr)\frac{\log^{\frac{\alpha_{\Lambda_{fg}}}{2}+1+\epsilon}x}{\log^{\frac{1}{2}}(\frac{V_2}{V_1})}.
\end{align*}
In particular if all the $\alpha$ and $\beta$ reach the minima in definitions \ref{alpha} and \ref{beta}, then the claim holds with $\epsilon=0$.}
\end{Teorema}

\begin{cor} \label{mainresultcor} \textit{Assuming the same hypotheses as in Theorem \ref{mainresult}
\begin{equation*}
\sum_{q \le Q} \frac{q}{\phi(q)} \sump_{\chi \bmod q} \max_{y \le x} \Bigl\vert \sum_{n \le y} \Lambda_{fg}(n) \chi(n) \Bigr\vert \ll ML
\end{equation*}
where $M$ is the main term and $L$ is the logarithmic term, precisely
\begin{align*}
 M = \max  \biggl\{ & U_1Q^2, \, x^{\theta_f}(U_0V_2)^{1-\theta_f}Q^{\frac{5}{2}},\, x^{\gamma_f}(U_0V_2)^{1-\gamma_f}Q^2, \, x, \, x^{\frac{1}{2}}Q^2, \\ & \frac{xQ}{U_0^{\frac{1}{2}}}, \,x^{\frac{1}{2}}U_1^{\frac{1}{2}}Q, \,V_2Q^{\frac{5}{2}}, \,\frac{xQ}{V_1^{\frac{1}{2}}}\biggr\}
\end{align*}
and
\begin{align*}
L = \max \biggl\{ & \log^{\beta_{\Lambda_{fg}}(0)+\epsilon}U_1, \, 
 \log^{\beta_{\mu_f}(\theta_f) + \beta_{\Lambda_{fg}}(\theta_f) + 1+\epsilon}(U_0V_2Q), \\ & \log^{\beta_{\mu_f}(\gamma_f) + \beta_{\Lambda_{fg}}(\gamma_f)+\epsilon}(U_0V_2), 
  \log^{\beta_f(0) + \beta_{\mu_f}(1) + \beta_{\Lambda_{fg}}(1)+\epsilon}(xU_0V_2), 
  \\ & \frac{\log^{\frac{\alpha_{\Lambda_{fg}}}{2}+2+\epsilon}(xU_1)}{\log^{\frac{1}{2}}(\frac{V_2}{V_1})}, \, g(x)\log^{\beta_{\mu_f}(0)+1+\epsilon}(V_2Q), \, g(x)\log^{\beta_{\mu_f}(1)+\epsilon}V_2, \\ & \frac{\log^{\frac{\alpha_{\Lambda_{fg}}}{2}+2+\epsilon}x}{\log^{\frac{1}{2}}(\frac{V_2}{V_1})}  \biggr\}.
\end{align*}}
\end{cor}

\section{Preparation for the proof}
First we prove a Lemma that guarantees us that if $D \in \mathscr{D}$ then $\beta_D(k)$ is bounded for all $k \in [0,1]$. 
\begin{Lemma}  \label{prop12}
\textit{ If $D \in \mathscr{D}$ then
\begin{equation*} 
\beta_D(k) \le \frac{\alpha_D}{2} + \mathds{1}(k).
\end{equation*}
}
\end{Lemma}
\begin{proof}
This follows easily using partial summation and the Cauchy-Schwarz inequality, 
\begin{align*}
\sum_{n \le x} \frac{\abs*{D(n)}}{n^{k}} & = \frac{1}{x^k} \sum_{n \le x} \abs*{D(n)} + k \int_1^x \Bigl(\sum_{n \le t} \abs*{D(n)} \Bigr) \frac{dt}{t^{k+1}} 
 \\ & \ll_k  \frac{1}{x^k} \Bigl( \sum_{n \le x} 1\Bigr)^{\frac{1}{2}} \Bigl(\sum_{n \le x} \abs*{D(n)}^2 \Bigr)^{\frac{1}{2}} +  \int_1^x \Bigl(\sum_{n \le t} 1\Bigr)^{\frac{1}{2}} \Bigl(\sum_{n \le t} \abs*{D(n)}^2 \Bigr)^{\frac{1}{2}} \frac{dt}{t^{k+1}} 
\\ & \ll_k x^{1-k} \log^{\frac{\alpha_D}{2}+\epsilon}x + \log^{\frac{\alpha_D}{2}+\epsilon}x \int_1^x \frac{dt}{t^k}, 
\end{align*}
for each $\epsilon>0$. So we have the claim distinguishing $k=1$ from the other cases.
\end{proof}
This is typically far from the best exponent, for example $\Lambda \in \mathscr{D}$ with $\alpha_{\Lambda} = 1$, Lemma \ref{prop12} provides us the bound $\beta_{\Lambda}(0) \le 1/2$ but the prime number Theorem claims that $\beta_{\Lambda}(0)=0$. Another example rises from Mertens' formula 
\begin{equation*}
\sum_{n \le x} \frac{\Lambda(n)}{n} = \log x + O(1)
\end{equation*}
and so $\beta_{\Lambda}(1) = 1$ but with the Lemma \ref{prop12} we can only obtain $\beta_{\Lambda}(1) \le 3/2$. However with our kind of generalization it can't be done better than Lemma \ref{prop12}, for example the function identically $1$ is in $\D$ with $\alpha_1=0$, $\beta_1(0)=0$ and $\beta_1(1)=1$.

As in the classic proof of the basic mean value Theorem we need a modified multiplicative large sieve inequality.
\begin{Teorema} \label{CRIV2}
\textit{Let $f_1, f_2$ be two arithmetic function, then
\begin{align*}
& \sum_{q \le Q} \frac{q}{\phi(q)} \sump_{\chi \bmod q} \max_{y \le M_1M_2} \Bigl\vert \sum_{n \le y} ({f_1}_{\le M_1} \star {f_2}_{\le M_2})(n) \chi(n) \Bigr\vert \\ &  \ll (Q^2+M_1)^{\frac{1}{2}}(Q^2+M_2)^{\frac{1}{2}} \Bigl( \sum_{n \le M_1}  \abs*{f_1(n)}^2 \Bigr)^{\frac{1}{2}} \Bigl(\sum_{n \le M_2} \abs*{f_2(n)}^2 \Bigr)^{\frac{1}{2}} \log{(M_1M_2)}.
\end{align*}}
\end{Teorema}
For the proof see Lemma 2 of \cite{Vaughan1980}. 
If we have to estimate sums like
\begin{align*}
& \sum_{q \le Q} \frac{q}{\phi(q)} \sump_{\chi \bmod q} \max_{y \le x} \Bigl\vert \sum_{n \le y} ({f_1}_{(N_1,M_1]} \star f_2)(n) \chi(n) \Bigr\vert,\end{align*}
with $f_1, f_2 \in \D$ and $M_1/N_1 \ll x$, using directly Theorem \ref{CRIV2} is not in general convenient. Indeed writing 
\begin{align*}
\max_{y \le x} \Bigl\vert \sum_{n \le y} ({f_1}_{(N_1,M_1]} \star f_2)(n) \chi(n) \Bigr\vert = \max_{y \le x} \Bigl\vert \sum_{n \le y} ({f_1}_{(N_1,M_1]} \star {f_2}_{\le \frac{x}{N_1}})(n) \chi(n) \Bigr\vert
\end{align*}
we obtain a bound like
\begin{align}
\nonumber & \ll \bigl(Q + M_1^{\frac{1}{2}}\bigr)\biggl(Q + \frac{x^{\frac{1}{2}}}{N_1^{\frac{1}{2}}}\biggr) M_1^{\frac{1}{2}} \frac{x^{\frac{1}{2}}}{N_1^{\frac{1}{2}}} \log^{\frac{\alpha_{f_1} + \alpha_{f_2}}{2}+1+\epsilon}x \\ \label{BEST} & = \biggl(x^{\frac{1}{2}}Q^{2}\Bigl(\frac{M_1}{N_1}\Bigr)^{\frac{1}{2}}+ x \frac{M_1}{N_1} + x^{\frac{1}{2}}Q\Bigl(\frac{M_1}{N_1}\Bigr)^{\frac{1}{2}}  \Bigl(M_1^{\frac{1}{2}} + \frac{x^{\frac{1}{2}}}{N_1^{\frac{1}{2}}}\biggr) \biggr)\log^{\frac{\alpha_{f_1} + \alpha_{f_2}}{2}+1+\epsilon}x.
\end{align}
Combining a dicotomic method with Theorem \ref{CRIV2} we can find a better bound when $\log M_1 \ll M_1/N_1$. 

\begin{Lemma} \label{CORFOND}
Given $f_1, f_2 \in \D$, $M_1, N_1$ such that $M_1/N_1 \ll x$ and $\epsilon>0$,
\begin{align}
\nonumber & \sum_{q \le Q} \frac{q}{\phi(q)} \sump_{\chi \bmod q} \max_{y \le x} \Bigl\vert \sum_{n \le y} ({f_1}_{(N_1,M_1]} \star f_2)(n) \chi(n) \Bigr\vert \\ 
& \label{CORFONDEQ} \ll 
\biggl(\bigl(x^{\frac{1}{2}}Q^2 + x\bigr)\log M_1 + x^{\frac{1}{2}}Q\Bigl(M_1^{\frac{1}{2}} + \frac{x^{\frac{1}{2}}}{N_1^{\frac{1}{2}}}\Bigr)\biggr)\log^{\frac{\alpha_{f_1} + \alpha_{f_2}}{2}+1+\epsilon}x. 
 \end{align}
\end{Lemma}

\begin{proof}
The estimate \eqref{BEST} is good when $M_1 \asymp N_1$
The idea is to split the interval $(N_1,M_1]$ in subintervals of the type $[T,2T]$ and then apply Theorem \ref{CRIV2} at each of this subintervals. For $T \le x$
\begin{align}
& \nonumber  \sum_{q \le Q} \frac{q}{\phi(q)} \sump_{\chi \bmod q} \max_{y \le x} \Bigl\vert \sum_{n \le y} ({f_1}_{(T,2T]} \star f_2)(n) \chi(n) \Bigr\vert \nonumber \\ & \nonumber = 
\sum_{q \le Q} \frac{q}{\phi(q)} \sump_{\chi \bmod q} \max_{y \le x} \Bigl\vert \sum_{n \le y} ({f_1}_{(T,2T]} \star {f_2}_{\le \frac{x}{T}})(n) \chi(n) \Bigr\vert
\\ \label{fondfact}
& \ll  
\bigl(Q + T^{\frac{1}{2}}\bigr)\biggl(Q + \frac{x^{\frac{1}{2}}}{T^{\frac{1}{2}}}\biggr) T^{\frac{1}{2}} \frac{x^{\frac{1}{2}}}{T^{\frac{1}{2}}} \log^{\frac{\alpha_{f_1} + \alpha_{f_2}}{2}+1+\epsilon}x \\
\label{proaa} &  =
\biggl(x^{\frac{1}{2}}Q^2 + x + x^{\frac{1}{2}}Q\Bigl(T^{\frac{1}{2}} + \frac{x^{\frac{1}{2}}}{T^{\frac{1}{2}}}\Bigr)\biggr)\log^{\frac{\alpha_{f_1} + \alpha_{f_2}}{2}+1+\epsilon}x. 
\end{align}
We choose $T=N_12^k$ by varying $k \in \mathscr{S} \subset \N$ such that 
\begin{equation*}
(N_1,M_1] \subset \bigcup_{k \in \mathscr{S}} [N_12^k,N_12^{k+1}]
\end{equation*}
and $\abs*{\mathscr{S}}$ is minimum. In general the inclusion will be proper, to avoid problems and to be able to use the triangle inequality we extend to zero $f_1$ in the external points to $ (N_1, M_1] $, i.e. we define $\tilde f_1 = {f_1}_{(N_1,M_1]} $.
Now using the triangle inequality
\begin{align*} 
\Bigl\vert \sum_{n \le y} ({f_1}_{(N_1,M_1]} \star f_2)(n) \chi(n) \Bigr\vert  & = \Bigl\vert \sum_{n \le y} (\tilde f_1 \star f_2)(n) \chi(n) \Bigr\vert 
\\  & \le \sum_{\substack{T = N_12^k \\ k \in \mathscr{S}}} \Bigl\vert \sum_{n \le y} (\tilde {f_1}_{(T,2T]} \star f_2)(n) \chi(n) \Bigr\vert.
\end{align*}
Since $T \in [N_1,2M_1]$, with \eqref{proaa} we can conclude
\begin{align*}
& \sum_{q \le Q} \frac{q}{\phi(q)} \sump_{\chi \bmod q} \max_{y \le x} \Bigl\vert \sum_{n \le y} ({f_1}_{(N_1,M_1]} \star {f_2})(n) \chi(n) \Bigr\vert \\ & \ll 
\biggl(\bigl(x^{\frac{1}{2}}Q^2 + x\bigr)\abs*{\mathscr{S}} + x^{\frac{1}{2}}Q\Bigl(M_1^{\frac{1}{2}} + \frac{x^{\frac{1}{2}}}{N_1^{\frac{1}{2}}}\Bigr)\biggr)\log^{\frac{\alpha_{f_1} + \alpha_{f_2}}{2}+1+\epsilon}x 
\end{align*}
and since $\abs*{\mathscr{S}} \ll \log M_1$, we obtain the claim.
\end{proof}

\begin{oss} \label{OSSFOND}
We remark that the previous Lemma is useful also when we have to estimate
\begin{align*}
& \sum_{q \le Q} \frac{q}{\phi(q)} \sump_{\chi \bmod q} \max_{y \le x} \Bigl\vert \sum_{n \le y} ({f_1}_{>N_1} \star {f_2}_{>N_2})(n) \chi(n) \Bigr\vert,
\end{align*}
indeed we can take $M_1= x/N_2$ and obtain the bound
\begin{align} \label{OSSFONDEQ}
\ll \biggl(\bigl(x^{\frac{1}{2}}Q^2 + x\bigr)\log{x} + xQ\Bigl(\frac{1}{N_1^{\frac{1}{2}}} + \frac{1}{N_2^{\frac{1}{2}}}\Bigr)\biggr)\log^{\frac{\alpha_{f_1} + \alpha_{f_2}}{2}+1+\epsilon}x 
 \end{align}
\end{oss}

\subsection{Weighted Vaughan's identity.}
We want to use a decomposition formula for $\Lambda_{fg}$ using a weight $\eta: \N \to \C$ such that $\eta(b) = 1$ for $b \le V_1$ as A. Sedunova did in \cite{Sedunova20182}.
We know the classic Vaughan's identity
\begin{align*}
\Lambda_{fg} & = {\Lambda_{fg}}_{\le U_1} -  {\Lambda_{fg}}_{\le U_1} \star {\mu_f}_{\le V_1} \star f + {\mu_f}_{\le V_1} \star g + {\Lambda_{fg}}_{>U_1} \star {\mu_f}_{> V_1} \star f \\ & = 
\Lambda_1 + \Lambda_2+ \Lambda_3+ \Lambda_4
\end{align*}
that follows from \eqref{convo1} and \eqref{convo2}, indeed
\begin{align*}
\Lambda_{fg} & =  {\Lambda_{fg}}_{\le U_1} +  \Lambda_{fg}  -  {\Lambda_{fg}}_{\le U_1} =
{\Lambda_{fg}}_{\le U_1} + \mu_f \star g -  {\Lambda_{fg}}_{\le U_1}  \star \mu_f \star f \\ & =
{\Lambda_{fg}}_{\le U_1} + {\mu_f}_{\le V_1} \star g + {\mu_f}_{> V_1} \star g  -  {\Lambda_{fg}}_{\le U_1}  \star {\mu_f}_{\le V_1} \star f - {\Lambda_{fg}}_{\le U_1}  \star {\mu_f}_{> V_1} \star f \\ & = 
{\Lambda_{fg}}_{\le U_1} - {\Lambda_{fg}}_{\le U_1}  \star {\mu_f}_{\le V_1} \star f +  {\mu_f}_{\le V_1} \star g   + {\mu_f}_{> V_1} \star (g - {\Lambda_{fg}}_{\le U_1} \star f)
\end{align*}
and we use that from \eqref{convo1} and \eqref{convo2} follows also \begin{equation} \label{convo3}
g = \Lambda_{fg} \star f.
\end{equation}
We claim that, more in general
\begin{Lemma} \label{weightedvaughan}
 \textit{For every $\eta : \N \to \C$ such that $\eta(b) = 1$ for every $b \le V_1$
\begin{align*}
\Lambda_{fg} & = {\Lambda_{fg}}_{\le U_1} -  {\Lambda_{fg}}_{\le U_1} \star (\mu_f \cdot \eta) \star f + (\mu_f \cdot \eta) \star g + {\Lambda_{fg}}_{>U_1} \star \bigl(\mu_f \cdot (1-\eta)\bigr) \star f \\ & = 
\Lambda_1' + \Lambda_2'+ \Lambda_3'+ \Lambda_4'.
\end{align*}}
\end{Lemma}
\begin{proof}
We observe, using essentially that $\eta(b) = 1$ for every $b \le V_1$,
\begin{align*}
\Lambda_1' & =\Lambda_1, \\
\Lambda_2' & =\Lambda_2 + {\Lambda_{fg}}_{\le U_1} \star (\mu_f \cdot \eta)_{>V_1} \star f, \\
\Lambda_3' & =\Lambda_3 - (\mu_f \cdot \eta)_{>V_1} \star g, \\
\Lambda_4' & =\Lambda_4 + {\Lambda_{fg}}_{>U_1} \star (\mu_f \cdot \eta)_{> V_1} \star f.
\end{align*}
It remains to show that the sum of the three remainders is equal to zero, but this is true since, from \eqref{convo3}
\begin{align*}
& {\Lambda_{fg}}_{\le U_1} \star (\mu_f \cdot \eta)_{>V_1} \star f  - (\mu_f \cdot \eta)_{>V_1} \star g+ {\Lambda_{fg}}_{>U_1} \star (\mu_f \cdot \eta)_{> V_1} \star f \\
& = (\mu_f \cdot \eta)_{>V_1} \star ( {\Lambda_{fg}}_{\le U_1} \star f - g +  {\Lambda_{fg}}_{>U_1} \star f) = 0. \qedhere
\end{align*} 
\end{proof}

\section{Main proof}
In the proof we denote with $\epsilon>0$ any small positive constant that rises from the definitions of $\alpha_D$ and $\beta_D(k)$ as infima; at the end we will still indicate with $\epsilon$ the maximum of the constant previously considered.
First we show, as R. C. Vaughan did in \cite{Vaughan1980}, that we can treat larger $Q$ more easily than smaller $Q$.
\subsection{The case $\mathbf{Q^2>x}$.}
We only use the modified multiplicative large sieve (Theorem \ref{CRIV2}) with $M_1=1, \, f_1(1)=1, \,  M_2=[x], \, f_2(n)=\Lambda_{fg}(n)$. We obtain
\begin{align*} \sum_{q \le Q} \frac{q}{\phi(q)} \sump_{\chi \bmod q} \max_{y \le x} \Bigl\vert \sum_{n \le y} \Lambda_{fg}(n) \chi(n) \Bigr \vert \ll (x^{\frac{1}{2}}Q+Q^2)\Bigl( \sum_{n \le x} \abs*{\Lambda_{fg}(n)}^2\Bigr)^{\frac{1}{2}} \log{x}.
\end{align*}
Using \ref{H2} and the definition of $\D$
\begin{align*}
\sum_{q \le Q} \frac{q}{\phi(q)} \sump_{\chi \bmod q} \max_{y \le x} \Bigl\vert \sum_{n \le y} \Lambda_{fg}(n) \chi(n) \Bigr \vert & \ll (xQ+x^{\frac{1}{2}}Q^2) \log^{\frac{\alpha_{\Lambda_{fg}}}{2}+1+\epsilon}x
\\ & \ll x^{\frac{1}{2}}Q^2 \log^{\frac{\alpha_{\Lambda_{fg}}}{2}+1+\epsilon}x
\end{align*}
since $Q^2>x$.

From now on we can assume $Q^2 \le x$.
We set four parameters \\ $U_0=U_0(x,Q)\le U_1=U_1(x,Q)$, $V_1=V_1(x,Q) < V_2=V_2(x,Q)$. 
Recalling Lemma \ref{weightedvaughan}, for any Dirichlet character $\chi \bmod q$ we can write
\begin{equation*}
\sum_{n \le y} \Lambda_{fg}(n) \chi(n) = \sum_{i=1}^4 \sum_{n \le y} \Lambda'_i(n) \chi(n) = \sum_{i=1}^4 S_i(y,\chi).
\end{equation*}
We prove the Theorem \ref{mainresult} by estimating each of the sums
\begin{equation*}
S_i(x,Q) = \sum_{q \le Q} \frac{q}{\phi(q)} \sump_{\chi \bmod q} \max_{y \le x} \abs*{S_i(y,\chi)}, \quad 1 \le i \le 4.
\end{equation*}
\subsection{The estimate for $\mathbf{S_1(x,Q)}$.}
Using hypothesis \ref{H2} and definition \eqref{beta} we obtain
\begin{equation*}
\abs*{S_1(y,\chi)} = \Bigl\vert \sum_{n \le \min \{U_1,y\}} \Lambda_{fg}(n) \chi(n) \Bigr\vert \le \sum_{n \le U_1} \abs*{\Lambda_{fg}(n)} \ll U_1 \log^{\beta_{\Lambda_{fg}}(0)+\epsilon}U_1,
\end{equation*} 
and so 
\begin{equation*}
S_1(x,Q) = \sum_{q \le Q} \frac{q}{\phi(q)} \sump_{\chi \bmod q} \max_{y \le x} \abs*{S_1(y,\chi)} \ll  U_1 Q^2 \log^{\beta_{\Lambda_{fg}}(0)+\epsilon}U_1.
\end{equation*}
\subsection{The estimate for $\mathbf{S_2(x,Q)}$.}
We recall the definition 
\begin{equation*}
S_2(y,\chi) = - \sum_{n \le y} \bigl( {\Lambda_{fg}}_{\le U_1} \star (\mu_f \cdot \eta) \star f\bigr)(n) \chi(n),
\end{equation*}
we split this sum into two parts
\begin{equation*}
S_2(y,\chi) = S_2'(y,\chi) + S_2''(y,\chi)
\end{equation*}
where
\begin{equation*}
S_2'(y,\chi) = - \sum_{n \le y} \bigl( {\Lambda_{fg}}_{\le U_0} \star (\mu_f \cdot \eta) \star f\bigr)(n) \chi(n),
\end{equation*}
and
\begin{equation*}
S_2''(y,\chi) = - \sum_{n \le y} \bigl( {\Lambda_{fg}}_{(U_0,U_1]} \star (\mu_f \cdot \eta) \star f\bigr)(n) \chi(n).
\end{equation*}
For $S_2'(y,\chi)$, using \ref{H4} and writing $n=abc$
\begin{align*}
\abs*{ S_2'(y,\chi)} & = \Bigl \vert \sum_{a \le U_0} \Lambda_{fg}(a) \chi(a) \sum_{b} \mu_f(b) \eta(b) \chi(b) \sum_{c \le \frac{y}{ab}} f(c) \chi(c) \Bigr\vert \\ & \ll 
 \sum_{a \le U_0} \abs*{\Lambda_{fg}(a)} \sum_{b \le V_2}  \abs*{\mu_f(b)} \Bigl \vert \sum_{c \le \frac{y}{ab}} f(c) \chi(c) \Bigr\vert,
\end{align*}
so that we can use hypothesis \ref{H3} to estimate the innermost sum for non-principal primitive characters $\chi \bmod q$. We get
\begin{align*}
\abs*{S_2'(y,\chi)} & \ll y^{\theta_f} q^{\frac{1}{2}} \log q \sum_{a \le U_0} \frac{\abs*{\Lambda_{fg}(a)}}{a^{\theta_f}} \sum_{b \le V_2} \frac{\abs*{\mu_f(b)}}{b^{\theta_f}} \\ & + 
y^{\gamma_f} \sum_{a \le U_0} \frac{\abs*{\Lambda_{fg}(a)}}{a^{\gamma_f}} \sum_{b \le V_2} \frac{\abs*{\mu_f(b)}}{b^{\gamma_f}}.
\end{align*}
Then, by using hypothesis \ref{H2} and using four times definition \eqref{beta}, we obtain
\begin{align*}
\abs*{S_2'(y,\chi)} & \ll y^{\theta_f} V_2^{1-\theta_f}q^{\frac{1}{2}} \log^{\beta_{\mu_f}(\theta_f)+1+\epsilon}(V_2q) \sum_{a \le U_0} \frac{\abs*{\Lambda_{fg}(a)}}{a^{\theta_f}} \\ & +  y^{\gamma_f}  V_2^{1-\gamma_f} \log^{\beta_{\mu_f}(\gamma_f)+\epsilon}V_2 \sum_{a \le U_0} \frac{\abs*{\Lambda_{fg}(a)}}{a^{\gamma_f}} \\ & \ll
y^{\theta_f}(U_0V_2)^{1-\theta_f} q^{\frac{1}{2}}  \log^{\beta_{\mu_f}(\theta_f)+\beta_{\Lambda_{fg}}(\theta_f)+1+\epsilon}(U_0 V_2q) \\ & + y^{\gamma_f} (U_0V_2)^{1-\gamma_f} \log^{\beta_{\mu_f}(\gamma_f)+\beta_{\Lambda_{fg}}(\gamma_f)+\epsilon}( U_0V_2).
\end{align*} Instead, for $\chi = \chi_0$ we have, using two times definition \ref{beta},
\begin{align*}
 \abs*{S_2'(y,\chi_0)} & \le \sum_{a \le U_0} \abs*{\Lambda_{fg}(a)} \sum_{b \le V_2}  \abs*{\mu_f(b)} \sum_{c \le \frac{y}{ab}} \abs*{f(c)} \\ &
\ll y \log^{\beta_f(0)+\epsilon}y \sum_{a \le U_0} \frac{\abs*{\Lambda_{fg}(a)}}{a} \sum_{b \le V_2}  \frac{\abs*{\mu_f(b)}}{b}
\\ & \ll y \log^{\beta_f(0) + \beta_{\mu_f}(1)+\epsilon}(y V_2) \sum_{a \le U_0} \frac{\abs*{\Lambda_{fg}(a)}}{a} 
\\ & \ll y \log^{\beta_f(0) + \beta_{\mu_f}(1) + \beta_{\Lambda_{fg}}(1)+\epsilon}(y \,  U_0 V_2).
\end{align*}
This implies that
\begin{align*}
S_2'(x,Q) & = \sum_{q \le Q} \frac{q}{\phi(q)} \sump_{\chi \bmod q} \max_{y \le x} \abs*{S_2'(y,\chi)} \\ & \ll x^{\theta_f} (U_0V_2)^{1-\theta_f}Q^{\frac{5}{2}} \log^{\beta_{\mu_f}(\theta_f) + \beta_{\Lambda_{fg}}(\theta_f) + 1+\epsilon}(U_0V_2Q) \\ & + x^{\gamma_f} (U_0V_2)^{1-\gamma_f}  Q^2 \log ^{\beta_{\mu_f}(\gamma_f) + \beta_{\Lambda_{fg}}(\gamma_f)+\epsilon}(U_0V_2) \\ & + x  \log^{\beta_f(0) + \beta_{\mu_f}(1) + \beta_{\Lambda_{fg}}(1)+\epsilon}(xU_0V_2).
\end{align*}
For $S_2''(y,\chi)$ we recall the definition
\begin{align*}
 S_2''(x,Q) = \sum_{q \le Q} \frac{q}{\phi(q)} \sump_{\chi \bmod q} \max_{y \le x} \Bigl\vert \sum_{n \le y} \bigl( {\Lambda_{fg}}_{(U_0,U_1]} \star (\mu_f \cdot \eta) \star f\bigr)(n) \chi(n) \Bigr\vert.
\end{align*} 
We want to use Lemma \ref{CORFOND}. We choose $f_1 = \Lambda_{fg}$, $f_2 = (\mu_f \cdot \eta) \star f$, $N_1 = U_0$ e $M_1 = U_1$. From hypothesis \ref{H4} we have $(\mu_f \cdot \eta) \star f \in \mathscr{D}$ with $\alpha_{(\mu_f \cdot \eta) \star f}=0$. Moreover we have stronger bounds than for other functions in $\mathscr{D}$, indeed we can include the denominator $1/\log(V_2/V_1)$ in \eqref{CORFONDEQ} since this does not depend on the upper limit of each partial sums. Finally we obtain
\begin{equation*}
S_2''(x,Q) \ll \biggl(\bigl(x^{\frac{1}{2}}Q^2 + x\bigr)\log U_1 + x^{\frac{1}{2}}Q\Bigl(U_1^{\frac{1}{2}} + \frac{x^{\frac{1}{2}}}{U_0^{\frac{1}{2}}}\Bigr)\biggr)\frac{\log^{\frac{\alpha_{\Lambda_{fg}}}{2}+1+\epsilon}x}{\log^{\frac{1}{2}}(\frac{V_2}{V_1})}
\end{equation*}

\subsection{The estimate for $\mathbf{S_3(x,Q)}$.}
We recall the definition
\begin{equation*}
S_3(x,Q) = \sum_{q \le Q} \frac{q}{\phi(q)} \sump_{\chi \bmod q} \max_{y \le x} \Bigl\vert 
\sum_{n \le y} \bigl( (\mu_f \cdot \eta) \star g\bigr)(n) \chi(n) \Bigr\vert.
\end{equation*}
We define a step function $\G : \R \to \R$ by $\G(t)=g(1)$ if $t \le 1$ and $\G(t) = g(n) - g(n-1)$ if $n-1 < t \le n$ for $n \ge 2$. Then we observe that $g(n) = \int_0^n \G(t) dt$ and that $\G$ is positive, since the function $g$ is positive and increasing from \ref{H1}. 
We write, by partial summation,
\begin{align*}
\abs*{S_3(y,\chi)} & = \Bigl \vert \mathop{\sum\sum}_{ab \le y} \mu_f(a) \eta(a) g(b)  \chi(ab) \Bigr\vert \\
& = \Bigl \vert \sum_{a \le V_2} \mu_f(a) \eta(a) \chi(a) \sum_{b \le \frac{y}{a}} \chi(b) \int_0^b \G(t) dt \Bigr \vert \\
& =  \Bigl \vert \sum_{a \le V_2} \mu_f(a) \eta(a) \chi(a) \int_0^{\frac{y}{b}} \sum_{t < b \le \frac{y}{a}} \chi(b) \, \G(t)  \, dt \Bigr \vert \\
& \le \int_0^y \G(t) \sum_{a \le V_2} \abs*{ \mu_f(a) \eta(a)} \Bigl\vert \sum_{t < b \le \frac{y}{a}} \chi(b) \Bigr \vert \, dt .
\end{align*} 
We can use the Pólya-Vinogradov inequality to estimate the inner sum for non-principal characters 
\begin{equation*}
\abs*{S_3(y,\chi)} \ll g(y) q^{\frac{1}{2}} \log q \sum_{a \le V_2} \abs*{\mu_f(a)\eta(a)}.
\end{equation*}
Moreover, using hypotheses \ref{H2} and \ref{H4} according with definition \eqref{beta}, we can write
\begin{equation*}
\abs*{S_3(y,\chi)} \ll g(y) V_2 q^{\frac{1}{2}} \log^{\beta_{\mu_f}(0)+1+\epsilon}(V_2q).
\end{equation*} 
For $\chi=\chi_0$, again using hypotheses \ref{H2} and \ref{H4} according with definition \eqref{beta} we can write
\begin{equation*}
\abs*{S_3(y,\chi_0)} \ll g(y) y \sum_{a \le V_2} \frac{\abs{\mu_f(a)}}{a} \ll g(y) y \log^{\beta_{\mu_f}(1)+\epsilon}V_2.
\end{equation*} We further obtain 
\begin{align*}
S_3(x,Q) = \sum_{q \le Q} \frac{q}{\phi(q)} \sump_{\chi \bmod q} \max_{y \le x} \abs*{S_3(y,\chi)} & \ll g(x) V_2 Q^{\frac{5}{2}} \log^{\beta_{\mu_f}(0)+1+\epsilon}(V_2Q) \\ & + 
g(x) x \log^{\beta_{\mu_f}(1)+\epsilon}V_2.
\end{align*}

\subsection{The estimate for $\mathbf{S_4(x,Q)}$.} 
We recall the definition
\begin{equation*}
S_4(x,Q)
= \sum_{q \le Q} \frac{q}{\phi(q)} \sump_{\chi \bmod q} \max_{y \le x} \Bigl\vert \sum_{n \le y} \bigl( {\Lambda_{fg}}_{>U_1} \star ( \mu_f \cdot (1-\eta)) \star f \bigr)(n) \chi(n) \Bigr\vert
\end{equation*}
We notice that $(1-\eta) = (1-\eta)_{>V_1}$ from \eqref{H4} and clearly
\begin{equation*}
{\Lambda_{fg}}_{>U_1} \star \bigl(( \mu_f \cdot (1-\eta))_{>V_1}  \star f \bigr)= {\Lambda_{fg}}_{>U_1} \star \bigl(( \mu_f \cdot (1-\eta)) \star f \bigr)_{>V_1}.
\end{equation*}
Moreover from \eqref{convo1} we have that
\begin{equation*}
\bigl(( \mu_f \cdot (1-\eta)) \star f \bigr)_{>V_1} = \bigl(\mathds{1} - (\mu_f \cdot \eta) \star f\bigr)_{>V_1} = - \bigl((\mu_f \cdot \eta) \star f\bigr)_{>V_1}.
\end{equation*}
So we now can use Remark \ref{OSSFOND} with $f_1=\Lambda_{fg}, f_2= -(\mu_f \cdot \eta) \star f, N_1 = U_1, N_2 = V_1$. In a similar way as we did for $S_2''(x,Q)$, we obtain
\begin{equation*}
S_4(x,Q) \ll \biggl(\bigl(x^{\frac{1}{2}}Q^2 + x\bigr)\log{x} + xQ\Bigl(\frac{1}{V_1^{\frac{1}{2}}} + \frac{1}{U_1^{\frac{1}{2}}}\Bigr)\biggr)\frac{\log^{\frac{\alpha_{\Lambda_{fg}}}{2}+1+\epsilon}x}{\log^{\frac{1}{2}}(\frac{V_2}{V_1})}.
\end{equation*}

\subsection{Completion of the proof.}
Putting these estimates together it holds that
\begin{align*}
S_1(x,Q) & \ll  U_1Q^2 \log^{\beta_{\Lambda_{fg}}(0)+\epsilon}U_1, \\
S_2'(x,Q) & \ll x^{\theta_f} (U_0V_2)^{1-\theta_f}Q^{\frac{5}{2}} \log^{\beta_{\mu_f}(\theta_f) + \beta_{\Lambda_{fg}}(\theta_f) + 1+\epsilon}(U_0V_2Q) \\ & + x^{\gamma_f} (U_0V_2)^{1-\gamma_f}  Q^2 \log ^{\beta_{\mu_f}(\gamma_f) + \beta_{\Lambda_{fg}}(\gamma_f)+\epsilon}(U_0V_2) \\ & + x  \log^{\beta_f(0) + \beta_{\mu_f}(1) + \beta_{\Lambda_{fg}}(1)+\epsilon}(xU_0V_2), \\
S_2''(x,Q) & \ll \biggl(\bigl(x^{\frac{1}{2}}Q^2 + x\bigr)\log U_1 + x^{\frac{1}{2}}Q\Bigl(U_1^{\frac{1}{2}} + \frac{x^{\frac{1}{2}}}{U_0^{\frac{1}{2}}}\Bigr)\biggr)\frac{\log^{\frac{\alpha_{\Lambda_{fg}}}{2}+1+\epsilon}x}{\log^{\frac{1}{2}}(\frac{V_2}{V_1})}, \\
S_3(x,Q) & \ll g(x) V_2 Q^{\frac{5}{2}} \log^{\beta_{\mu_f}(0)+1+\epsilon}(V_2Q) +  g(x) x \log^{\beta_{\mu_f}(1)+\epsilon}V_2,
 \\
S_4(x,Q) & \ll \biggl(\bigl(x^{\frac{1}{2}}Q^2 + x\bigr)\log{x} + xQ\Bigl(\frac{1}{V_1^{\frac{1}{2}}} + \frac{1}{U_1^{\frac{1}{2}}}\Bigr)\biggr)\frac{\log^{\frac{\alpha_{\Lambda_{fg}}}{2}+1+\epsilon}x}{\log^{\frac{1}{2}}(\frac{V_2}{V_1})}.
\end{align*}
This gives the claim. 
We must be careful with $g(x)$: in Corollary \ref{mainresultcor} we have chosen to incorporate it in $L$ since in the benchmark case we have $g(x) = \log x$ but in general we have to know its growth and understand if it is better to integrate it in $L$ or in $M$.

\section{The choice of $U_0, \, U_1, \, V_1, \, V_2$}
Since we have the trivial bounds \eqref{triv1} and \eqref{triv2} we would like to find four parameters such that $M=o(xQ^2)$. We also note that there is symmetry in $M$ with $U_1$ and $V_1$, so we can always assume $U_1=V_1$ and so the scale is $U_0 \le U_1=V_1 < V_2$. 
Assuming that we can choose $U_0 \le U_1=V_1 < V_2 \ll x$, with $V_2/U_1 \gg x^{c}$ for some $c>0$ then $L \ll \log^{l+\epsilon}x$, where
\begin{align*}
l = \max \Bigl\{ & \beta_{\Lambda_{fg}}(0), \, \beta_{\mu_f}(\theta_f) + \beta_{\Lambda_{fg}}(\theta_f) + 1, \,  \beta_{\mu_f}(\gamma_f) + \beta_{\Lambda_{fg}}(\gamma_f), \\ & \, \beta_f(0) + \beta_{\mu_f}(1)+\beta_{\Lambda_{fg}}(1), \, \frac{\alpha_{\Lambda_{fg}}+3}{2}, \, \beta_{\mu_f}(0)+1, \, \beta_{\mu_f}(1) \Bigr\}.
\end{align*}
In view of Lemma \ref{prop12} we have the rough bound for $l$
\begin{align*}
l & \le \max \Bigl\{ \frac{\alpha_{\Lambda_{fg}}}{2}, \, \frac{\alpha_{\mu_f}+\alpha_{\Lambda_{fg}}}{2} + 2 \cdot \mathds{1}(\theta_f)+ 1, \, \frac{\alpha_{\mu_f}+\alpha_{\Lambda_{fg}}}{2} + 2 \cdot \mathds{1}(\gamma_f), \\ & \hspace{1.4cm} \frac{\alpha_{f}+ \alpha_{\mu_f}+\alpha_{\Lambda_{fg}}}{2} +  2, \,  \frac{\alpha_{\Lambda_{fg}}+3}{2} \, , \frac{\alpha_{\mu_f}}{2} + 1
\Bigr\} \\ & \le \max \Bigl\{ \frac{\alpha_f + \alpha_{\mu_f} + \alpha_{\Lambda_{fg}}}{2}+2, \,  \frac{\alpha_{\mu_f}+\alpha_{\Lambda_{fg}}}{2} + 2 \cdot \mathds{1}(\theta_f)+ 1 \Bigr\}.\end{align*}
\section{The classic case}
In R. C. Vaughan's basic mean value Theorem we treat
 \begin{equation*}
f=1, \quad g=\log,\quad \mu_f=\mu, \quad \Lambda_{fg}=\Lambda .\end{equation*} We have
\begin{align*}
& \sum_{n \le x} 1 = x + O(1), \hspace{2cm} \sum_{n \le x} \bigl\vert \Lambda(n)\bigr\vert^2 = x \log x + O(x), \\
& \sum_{n \le x} \Lambda(n) = x + O\biggl(\frac{x}{\log x}\biggr), \hspace{0.55cm} \sum_{n \le x} \frac{1}{n} = \log x + O(1), \\
& \sum_{n \le x} \frac{\Lambda(n)}{n} = \log x + O(1).
\end{align*}
In our notation we obtain
\begin{align*}
&\beta_{\mu}(0) = \beta_{1}(0) = 0 , \quad & \alpha_{\Lambda} = 1, \\
&\beta_{\Lambda}(0)  = 0, \quad & \beta_{\mu}(1) = 1, \\
&\beta_{\Lambda}(1) = 1;
\end{align*} 
all these values clearly are minima.
From Pólya-Vinogradov inequality we have 
\begin{equation*}
\theta_{1} = \gamma_{1} = 0.
\end{equation*}
To satisfy \ref{H4} we recall an estimate due to M. B. Barban and P. P. Vehov \cite{BarbanVehov1968} related to Selberg's sieve (see S. Graham for a stronger result \cite{Graham1978}).
For each $1 \le V_1 < V_2$ it holds
\begin{equation} \label{barb}
\sum_{n=1}^V \Bigl\vert \bigl(\mu \cdot \eta) \star 1\bigr)(n) \Bigr\vert^2 \ll \frac{V} {\log (\frac{V_2}{V_1})},
\end{equation} 
where
\begin{equation} \label{eta}
\eta(b) = \begin{cases}
1 & b \le V_1, \\
\frac{ \log{(\frac{V_2}{b})}}{\log{(\frac{V_2}{V_1})}} & V_1 < b \le V_2, \\
0 & b > V_2.
\end{cases}
\end{equation} 
As a Corollary of Theorem \ref{mainresult} we have the main result of \cite{Sedunova20182}.
\begin{cor} \label{sedunova} \textit{ For each $U_0 = U_0(x,Q) \le U_1 = U_1(x,Q)$, $V_1 = V_1(x,Q) < V_2 =  V_2(x,Q)$ it holds
\begin{equation*}
\sum_{q \le Q} \frac{q}{\phi(q)} \sump_{\chi \bmod q} \max_{y \le x} \Bigl\vert \sum_{n \le y} \Lambda(n) \chi(n) \Bigr\vert \ll M L,
\end{equation*}
where $M$ and $L$ are
\begin{align*}
 M = \max  \biggl\{ & U_1Q^2, \, (U_0V_2)Q^{\frac{5}{2}}, \, x, \, x^{\frac{1}{2}}Q^2, \, \frac{xQ}{U_0^{\frac{1}{2}}}, \,x^{\frac{1}{2}}U_1^{\frac{1}{2}}Q, \, \frac{xQ}{V_1^{\frac{1}{2}}}
 \biggr\}
\end{align*}
and
\begin{align*}
L = \max \biggl\{ &
 \log (U_0V_2Q), \, 
  \log^2(xU_0V_2),  \frac{\log^{\frac{5}{2}}(xU_1)}{\log^{\frac{1}{2}}(\frac{V_2}{V_1})}, 
 \,\log^2(xV_2Q), \, \frac{\log^{\frac{5}{2}}x}{\log^{\frac{1}{2}}(\frac{V_2}{V_1})}  
    \biggr\}.
\end{align*}}
\end{cor} 
With this result A. Sedunova, in \cite{Sedunova20182}, obtained 
\begin{Teorema} \label{BMVTsed} \textit{
For each $\epsilon>0$,
\begin{equation*}
\sum_{q \le Q} \frac{q}{\phi(q)} \sump_{\chi \bmod q} \max_{y \le x} \Bigl\vert \sum_{n \le y} \Lambda(n) \chi(n) \Bigr\vert \ll \bigl(x + x^{\frac{13}{14}+\epsilon}Q + x^{\frac{1}{2}}Q^2 \bigr) \log^2x.
\end{equation*}}
\end{Teorema}
This follows taking for $Q \in [x^{3/7+\epsilon}, x^{1/2}]$
\begin{equation*}
U_0 = x^{\frac{4}{7}-\epsilon}Q^{-1}, \, \quad U_1 = V_1 = x^{\frac{4}{7}}Q^{-1}, \quad V_2 = x^{\frac{4}{7}+\frac{5\epsilon}{2}}Q^{-1};
\end{equation*}
while for $Q \in [1,x^{3/7+\epsilon}]$
\begin{equation*}
U_0 = x^{\frac{1}{7}-\epsilon}, \quad U_1 = V_1 = x^{\frac{1}{7}}, \quad V_2 = x^{\frac{1}{7}+\frac{\epsilon}{2}}.
\end{equation*}
We remark that the exponent $13/14$ is optimal here, i.e. searching for the minimal $A>0$ such that for each $\epsilon>0$ it holds
\begin{equation*}
\sum_{q \le Q} \frac{q}{\phi(q)} \sump_{\chi \bmod q} \max_{y \le x} \Bigl\vert \sum_{n \le y} \Lambda(n) \chi(n) \Bigr\vert \ll \bigl(x + x^{A+\epsilon}Q + x^{\frac{1}{2}}Q^2 \bigr) \log^2x
\end{equation*}
then one can show that only using Corollary \ref{sedunova} it cannot be taken $A < 13/14$.

\section{Application to the generalized von Mangoldt function}
The generalized von Mangoldt function is defined as
\begin{equation*}
\Lambda_k = \mu \star \log^k \, 
\end{equation*}
 for $k \in \N$. One can show the recursive relation
 \begin{equation*}
 \Lambda_{k+1} = \Lambda_k \cdot \log + \Lambda \star \Lambda_k
 \end{equation*}
 and so, in particular, $\Lambda_k(n) \ge 0$.
In \cite{Levinson1965} it is shown that 
\begin{equation} \label{levin}
\sum_{n \le x} \Lambda_k(n) \sim k x \log^{k-1}x.
\end{equation}
From the Möbius inversion formula it holds
 \begin{equation*}
 \log^k = \Lambda_k \star 1
 \end{equation*}
 and so $\Lambda_k(n) \le (\log n)^k$.
We can easily derive from this and \eqref{levin} that
\begin{equation*}
\sum_{n \le x} \bigl\vert \Lambda_k(n)\bigr\vert^2 \ll_k x \log^{2k-1}x,
\end{equation*}
moreover, by partial summation and \eqref{levin} we have
\begin{align*}
\sum_{n \le x} \frac{\Lambda_k(n)}{n} & = k \log^{k-1}x + o(\log^{k-1}x) + \int_1^x \frac{k\log^{k-1}t}{t}dt
\sim \log^k x.
\end{align*}
Finally, $\Lambda_k \in \mathscr{D}$ with $\beta_{\Lambda_k}(1) = k$, $\beta_{\Lambda_k}(0) = k-1$ and $\alpha_{\Lambda_k} \le 2k-1$.
We can use the main Theorem \ref{mainresult} with $f=1$, $g=\log^k$ and then proceeding with the same choice of $U$ and $V$ as in \cite{Sedunova20182} to obtain 
\begin{Teorema} \textit{ For each $k \in \N$, $\epsilon>0$ it holds
\begin{equation*}
\sum_{q \le Q} \frac{q}{\phi(q)} \sump_{\chi \bmod q} \max_{y \le x} \Bigl\vert \sum_{n \le y} \Lambda_k(n) \chi(n) \Bigr\vert \ll_k \bigl(x + x^{\frac{13}{14}+\epsilon}Q + x^{\frac{1}{2}}Q^2 \bigr) \log^{k+1}x.
\end{equation*}}
\end{Teorema}
This is clearly a generalization of Theorem \ref{BMVTsed}.

\section{Remark on hypothesis (H4)}
We remark that in the classic case it holds something stronger than \eqref{barb} as S. Graham has shown in \cite{Graham1978}. We too can assume a stronger hypothesis than \ref{H4}. \\
{\fontfamily{arabic}\selectfont
(\hspace{0.02cm}H\hspace{0.02cm}4')}  \textit{
For each $1 \le V_1 < V_2 $ it holds
\begin{align*}
\sum_{n=1}^V \bigl(\Gamma_1 \star f\bigr)(n) \bigl(\Gamma_2 \star f\bigr)(n) = V \log V_1 + O(V)
\end{align*} 
where
\begin{equation*}
\Gamma_i(b) = \begin{cases} \mu_f(b) \log \left( \frac{V_i}{b} \right) & b \le V_i, \\ 
0 & b > V_i.
\end{cases}
\end{equation*}}
This implies \ref{H4}. Indeed we consider the same $\eta$ as in \eqref{eta}, and observe that $\eta \cdot \mu_f = (\Gamma_2-\Gamma_1)/\log (V_2/V_1)$,
so we can write 
\begin{align*}
 \log^2 \Bigl( \frac{V_2}{V_1} \Bigr) \sum_{n=1}^V \Bigl\vert  \bigl((\mu_f \cdot \eta) \star f \bigr)(n) \Bigr\vert^2 & =  \sum_{n=1}^V (\Gamma_1 \star f)^2(n) + \sum_{n=1}^V (\Gamma_2 \star f)^2(n) \\ &
- 2\sum_{n=1}^V (\Gamma_1 \star f)(n) (\Gamma_2 \star f)(n).
\end{align*}
Now we apply three times (\hspace{0.02cm}H\hspace{0.02cm}4') to obtain
\begin{align*}
\log^2 \Bigl( \frac{V_2}{V_1} \Bigr) \sum_{n=1}^V \Bigl\vert \bigl((\mu_f \cdot \eta) \star f \bigr)(n) \Bigr\vert^2 & = V \log V_1 + V \log V_2 
- 2 V \log V_1 + O(V)
\end{align*}
and so \ref{H4}.
\subsection*{Acknowledgements}
The author is grateful to his supervisor Alessandro Zaccagnini for his help and his revision. 

\begin{small}
\bibliographystyle{plain}
\bibliography{bibtex}{}
\end{small}
\noindent Dipartimento di Scienze, Matematiche, Fisiche e Informatiche \\ \hspace{0.35cm} Università di Parma \\
\hspace{0.35cm} Parco Area delle Scienze 53/a \\
\hspace{0.35cm} 43124 Parma, Italia \\
\hspace{0.35cm} email (MF): {\fontfamily{qcr}\selectfont
matteo.ferrari14@studenti.unipr.it
} 

\end{document}